\documentclass[11pt]{article}
\usepackage{amssymb,amsmath,amsthm,latexsym}
\title{Veronese subspace codes}
\author{Antonio Cossidente\\ Dipartimento di Matematica Informatica ed Economia \\ Universit\`a della Basilicata\\
 Contrada Macchia Romana\\ I-85100 Potenza\\ Italy\\antonio.cossidente@unibas.it\\
 Francesco Pavese\\Dipartimento di Matematica Informatica ed Economia\\ Universit\`a della Basilicata\\
 Contrada Macchia Romana\\ I-85100 Potenza\\ Italy\\francesco.pavese@unibas.it}
\date{}

\begin{document}
\maketitle

 \newpage\noindent
 {\bf Proposed Running Head:} Veronese subspace codes
 \vspace{2cm}\par\noindent {\bf Corresponding
 Author:}\\Antonio Cossidente\\
 Dipartimento di Matematica Informatica ed Economia\\ Universit\`a della Basilicata\\
 Contrada Macchia Romana\\ I-85100 Potenza\\ Italy\\antonio.cossidente@unibas.it
\newpage
\newtheorem{theorem}{Theorem}[section]
\newtheorem{lemma}[theorem]{Lemma}
\newtheorem{conj}[theorem]{Conjecture}
\newtheorem{remark}[theorem]{Remark}
\newtheorem{cor}[theorem]{Corollary}
\newtheorem{prop}[theorem]{Proposition}
\newtheorem{defin}[theorem]{Definition}
\newtheorem{result}[theorem]{Result}

 \def\runningheadeven{Subspace codes}
\def\runningheadodd{Antonio Cossidente and Francesco Pavese}

\newcommand{\Prf}{\noindent{\bf Proof}.\quad }
\renewcommand{\labelenumi}{(\arabic{enumi})}


\def\bA{\mathbf A}
\def\bE{\mathbf E}
\def\bF{\mathbf F}
\def\bG{\mathbf G}
\def\bP{\mathbf P}
\def\bN{\mathbf N}
\def\bZ{\mathbf Z}
\def\bL{\mathbf L}
\def\bQ{\mathbf Q}
\def\bU{\mathbf U}
\def\bV{\mathbf V}
\def\bW{\mathbf W}
\def\bX{\mathbf X}
\def\bY{\mathbf Y}
\def\cC{\mathcal C}
\def\cD{\mathcal D}
\def\cE{\mathcal E}
\def\cF{\mathcal F}
\def\cG{\mathcal G}
\def\cH{\mathcal H}
\def\cL{\mathcal L}
\def\cM{\mathcal M}
\def\cK{\mathcal K}
\def\cO{\mathcal O}
\def\cP{\mathcal P}
\def\cX{\mathcal X}
\def\cY{\mathcal Y}
\def\cU{\mathcal U}
\def\cV{\mathcal V}
\def\cT{\mathcal T}
\def\cR{\mathcal R}
\def\cS{\mathcal S}
\def\cW{\mathcal W}
\def\cFR{\mathcal{Fr}}
\def\PG{{\rm PG}}
\def\PGL{{\rm PGL}}
\def\GF{{\rm GF}}
\def\AG{{\rm AG}}
\def\GL{{\rm GL}}
\def\PGO{{\rm PGO}}

\def\ps@headings{
 \def\@oddhead{\footnotesize\rm\hfill\runningheadodd\hfill\thepage}
 \def\@evenhead{\footnotesize\rm\thepage\hfill\runningheadeven\hfill}
 \def\@oddfoot{}
 \def\@evenfoot{\@oddfoot}
}

\begin{abstract}
Using the geometry of quadrics of a projective plane $\PG(2,q)$
a family of $(6,q^3(q^2-1)(q-1)/3+(q^2+1)(q^2+q+1),4;3)_q$ constant dimension subspace codes is constructed.
\end{abstract}

    \section{Introduction}
 Let $V$ be an $n$--dimensional vector space over $\GF(q)$, $q$ any prime power. The set $S(V)$ of all subspaces of $V$, or subspaces of the projective space $\PG(V)$, forms a metric space with respect to the {\em subspace distance} defined by $d_s(U,U')=\dim (U+U')- \dim(U\cap U')$. In the context of subspace codes, the main problem is to determine the largest possible size of codes in the space $(S(V),d_s)$ with a given minimum distance, and to classify the corresponding optimal codes. The interest in these codes is a consequence of the fact that codes in the projective space and codes in the Grassmannian over a finite field referred to as subspace codes and constant--dimension codes, respectively, have been proposed for error control in random linear network coding.
An $(n,M,d;k)_q$ constant--dimension subspace code (CDC) is a set $\cal C$ of $k$--subspaces of $V$ with $\vert{\cal C}\vert =M$ and minimum subspace distance $d_s({\cal C})=\min\{d_s(U,U') \vert U,U'\in {\cal C}, U\ne U' \}=d$. The smallest open constant--dimension case occurs when $n = 6$ and $k = 3$. From a projective geometry point of view it translates in the determination of the maximum number of planes in $\PG(5,q)$ mutually intersecting in at most one point.
In \cite{HKK}, the authors show that the maximum size of a binary subspace code of packet length $n=6$, minimum subspace distance $d=4$ and constant dimension $k=3$ is $M=77$. Therefore the maximum number of planes in $\PG(5,2)$ mutually intersecting in at most one point is $77$. In the same paper, the authors, with the aid of a computer, classify all $(6,77,4;3)_2$ subspace codes into $5$ isomorphism types \cite[Table 6]{HKK} and a computer--free construction of one isomorphism type \cite[Table 6, A]{HKK} is provided. This last isomorphism type is then generalized to any $q$ providing a family of $(6,q^6+2q^2+2q+1,4;3)_q$ subspace codes \cite[Lemma 12]{HKK}. In \cite{CP1} the authors provided a construction of families of $(6,q^6+2q^2+2q+1,4;3)_q$ subspace codes potentially including the infinite family constructed in \cite{HKK}.

In this paper we construct a family of $(6,q^3(q^2-1)(q-1)/3+(q^2+1)(q^2+q+1),4;3)_q$ CDC.
Our approach is purely geometric and the construction relies on the geometry of quadrics of a projective plane $\PG(2,q)$. More precisely, we use the correspondence between quadrics of $\PG(2,q)$ and points of $\PG(5,q)$. In this setting, we show that a special net of conics (circumscribed bundle) yields a $(6,q^3(q^2-1)(q-1)/3,4;3)_q$ CDC admitting the linear group $\PGL(3,q)$ as an automorphism group. Although the size of such a code asymptotically reaches the theoretical upper bound of a  $(6,M,4;3)_q$ CDC \cite{HKK}, it turns out that it can be enlarged.
This is done in the second part of the paper where we are able to find a set of further $(q^2+1)(q^2+q+1)$ planes of $\PG(5,q)$
mutually intersecting in at most one point and extending the previous code. The $(6,q^3(q^2-1)(q-1)/3+(q^2+1)(q^2+q+1),4;3)_q$ CDC so obtained admits the normalizer of a Singer cyclic group of $\PGL(3,q)$ as an automorphism group.

	\section{The Veronese embedding}	
	
Let $\PG(2,q)$ the Desarguesian projective plane of order $q$.  A {\em quadric} of $\PG(2,q)$ is the locus of zeros of a quadratic polynomial, say $a_{11}X_1^2+a_{22}X_2^2+a_{33}X_3^2+a_{12}X_1X_2+a_{13}X_1X_3+a_{23}X_2X_3$. There are six parameters associated to such a curve and hence the set of quadrics of $\PG(2,q)$ forms a $5$--dimensional projective space. There exist four kinds of quadrics in $\PG(2,q)$, three of which are {\em degenerate} (splitting into lines, which could be in the plane $\PG(2,q^2)$) and one of which is {\em non--degenerate} \cite{JWPH1}.

The {\em Veronese map v} defined by
$$
a_{11}X_1^2+a_{22}X_2^2+a_{33}X_3^2+a_{12}X_1X_2+a_{13}X_1X_3+a_{23}X_2X_3\mapsto (a_{11},a_{22},a_{33},a_{12},a_{13},a_{23}),
$$
is the correspondence between plane quadrics and the points of $\PG(5,q)$,

The quadrics in $\PG(2,q)$ are:

\begin{itemize}
\item[1.] $q^2+q+1$ repeated lines;
\item[2.] $(q^2+q+1)(q+1)q/2$ quadrics consisting of two distinct lines of $\PG(2,q)$ (bi--lines);
\item[3.] $(q^2+q+1)(q-1)q/2$ quadrics consisting of two distinct conjugate lines of $\PG(2,q^2)$ (imaginary bi--lines).
\item[4.] $q^5-q^2$ non--degenerate quadrics (conics).
\end{itemize}

We will say that a bi--line or an imaginary bi--line is {\em centered} at $A$ if its lines meet in the point $A$.

When $q$ is even, all tangent lines to a conic $\cal C$ pass through a point of $\PG(2,q)$ called the {\em nucleus} of $\cal C$.

It is not difficult to see that a quadric of $\PG(2,q)$ is degenerate if and only if its parameters satisfy the polynomial
$$P_1:=X_4X_5X_6+X_1X_6^2+X_2X_5^2+X_3X_4^2,$$ when $q$ is even and
$$P_2:=X_1X_2X_3+2X_4X_5X_6+X_1X_6^2+X_2X_5^2+X_3X_4^2,$$ when $q$ is odd.

Notice that from \cite[Theorem 25.1.3]{HT} the image of the Veronese map {\em v} is the dual of the image of the map $\zeta$ defined in \cite[p. 146]{HT}.

The group $G:=\PGL(3,q)$ acts on $\PG(2,q)$ and so it also acts naturally on the plane quadrics, and hence also on $\PG(5,q)$.
The four sets of quadrics described above are $G$--orbits. With a slight abuse of notation we will denote by $G$ the group $\PGL(3,q)$
acting on $\PG(5,q)$.
We will denote by ${\cal O}_i$, $i=1,2,3,4$ the images under {\em v} in $\PG(5,q)$ of the four types of quadrics, respectively.
It turns out that ${\cal O}_1$ is the {\em Veronese surface} when $q$ is odd and a plane (called {\em degenerate Veronese surface})
when $q$ is even.  The orbits ${\cal O}_i$, $i=1,2,3$, partition the cubic hypersurface ${\cal S}$ of $\PG(5,q)$ with equation $P_1=0$ when $q$ is even and with equation $P_2=0$ when $q$ is odd.

It should be noted that under the map {\em v} a $k$--dimensional linear system of quadrics of $\PG(2,q)$ corresponds to a $(k-1)$--dimensional projective subspace of $\PG(5,q)$. This means that pencils, nets and webs of quadrics, are represented by lines, planes and solids of $\PG(5,q)$, respectively.

Let us fix a point $A$ of $\PG(2,q)$. The $q+1$ lines passing through $A$ considered as repeated lines,  the $q(q+1)/2$ bi--lines centered at $A$ and the $q(q-1)/2$ imaginary bi--lines centered at $A$ form a net that under the Veronese map {\em v} corresponds to a plane $\pi_A$ contained in $\cal S$ and meeting ${\cal O}_1$ at $q+1$ points forming either a conic ($q$ odd) or a line ($q$ even).
Hence, there is a set $\cal N$ of $q^2+q+1$ such planes. Also, through a point $P\in {\cal S}\setminus {\cal O}_1$ there passes exactly one plane of $\cal N$ whereas through a point of ${\cal O}_1$ there pass $q+1$ planes of $\cal N$ . It follows that two distinct planes in $\cal N$ meet in a point of $\cO_1$.

Let us fix two distinct points of $\PG(2,q)$, say $A$ and $B$. Let $\ell$ be the line joining $A$ and $B$. There are $q^2+q$ bi--lines of $\PG(2,q)$ containing $\ell$. These bi--lines together $\ell$ (considered as a repeated line) form a  net that under the Veronese map {\em v} corresponds to a plane $\pi_{\ell}$ contained in $\cal S$ and tangent to ${\cal O}_1$. Hence, there is a set $\cal T$ of $q^2+q+1$ such planes. Also, through a point $P\in{\cal O}_2$ there pass exactly two planes of $\cal T$ whereas through a point $P\in{\cal O}_1$ there passes exactly one plane of $\cal T$. It follows that two distinct planes in $\cal T$ meet in a point of $\cO_2$.

	\section{Circumscribed bundles}

There exists a collection of $q^2+q+1$ conics in $\PG(2,q)$ that mutually intersect in exactly one point, and hence serve as the lines of another projective plane on the points of $\PG(2,q)$.
Such a collection of conics is called a {\em projective bundle} of $\PG(2,q)$. For more details on projective bundles, see \cite{BBEF}.

\begin{remark}
{\rm Let us embed $\PG(2,q)$ into $\PG(2,q^3)$, and let $\sigma$ be the period $3$ collineation of $\PG(2,q^3)$ fixing $\PG(2,q)$. Let us fix a triangle $T$ of vertices $P$, $P^\sigma$, $P^{\sigma^2}$ in $\PG(2,q^3)$. Up to date, the known types of projective bundles are as follows \cite{DGG}, \cite{DGG1}:

	\begin{itemize}
\item[1.] {\em circumscribed bundle} consisting of all conics of $\PG(2,q)$ containing the vertices of $T$. This exists for all $q$;

\item[2.] {\em inscribed bundle} consisting of all conics of $\PG(2,q)$ that are tangent to the three sides of $T$. This exists for all odd $q$;

\item[3.] {\em self--polar bundle} consisting of all conics of $\PG(2,q)$ with respect to which $T$ is self--polar. This exists for all odd $q$.
	
	\end{itemize}
From \cite{BBEF} the conics of a circumscribed bundle form a net.}	
	\end{remark}
	
	\begin{remark}
{\rm  A cyclic group of $G$ permuting points (lines) of $\PG(2,q)$ in a single orbit is called a {\em Singer cyclic group} of $G$.
A generator of a Singer cyclic group is called a {\em Singer cycle}.

A Singer cyclic group of $G$ has order $q^2+q+1$ and its normalizer in $G$ turns out to be a metacyclic group of order $3(q^2+q+1)$. For more details, see \cite{huppert}.}
	\end{remark}

	\begin{remark}		
{\rm All these projective bundles are invariant under the normalizer of a Singer cyclic group of $G$.}		
	\end{remark}

Let ${\cal B}$ be a circumscribed bundle of $\PG(2,q)$.  We will need the following result, which extends \cite[Lemma 3.2]{CP1}.

\begin{lemma}\label{bundle}
     Consider two distinct conics $C_0,C_\infty$ of a circumscribed bundle $\cal B$. If $q$ is even, their nuclei are distinct. If $q$ is odd, for a point $P\in\PG(2,q)$ the polar lines of $P$ with respect to $C_0$ and $C_\infty$ are distinct.
\end{lemma}
\begin{proof}
If $q$ is even or if $q$ is odd and $P=C_0\cap C_\infty$ then the result follows from \cite[Lemma 3.2]{CP1}. Assume that $q$ is odd and $P\ne C_0\cap C_\infty$. Let $r_0$ and $r_\infty$ be the polar lines of $P$ with respect to $C_0$ and $C_\infty$, respectively.  By way of contradiction let $r_0=r_\infty$. If  $P\in C_0$ then $P\in r_0$ but $P\not\in r_\infty$,  a contradiction.

Let $A_0,A_{\infty}$, be the symmetric $3\times 3$ matrices associated to $C_0$ and $C_{\infty}$, respectively. Let ${\cal F}=\{C_\lambda,\lambda\in\GF(q)\cup\{\infty\} \}$ be the pencil generated by $C_0$ and $C_\infty$. We have that the quadrics of $\cal F$ are the conics (non degenerate quadrics) of $\cal B$ through $C_0\cap C_\infty$ and they cover all points of $\PG(2,q)$ . Let $\perp_\lambda$ denote the polarity associated with the conic in $C_\lambda\in {\cal F}$.
The product $\perp_0\perp_\infty$ is then a projectivity of $\PG(2,q)$ fixing $P$ whose associated matrix is $A_0^{-t}A_\infty$, where $t$ denotes transposition. In other terms $(A_0^{-t}A_ {\infty})(P^t)=\rho P^t$, for some $\rho\in\GF(q)\setminus\{0\}$. Analogously, $\perp_0\perp_{\lambda}$ is a projectivity of $\PG(2,q)$ whose associated matrix is $A_0^{-t}(A_0+\lambda A_{\infty})$ and fixing $P$.
Indeed, $(A_0^{-t}(A_0+\lambda A_{\infty}))(P^t)=(I+\lambda A_0^{-t}A_{\infty})(P^t)=P^t+\lambda\rho P^t=(1+\lambda\rho)(P^t)$.
It turns out that $(P^{\perp_{\lambda}})^{\perp_0}=P$ if and only if $P^{\perp_{\lambda}}=r_0$ if and only if $r_0^{\perp_{\lambda}}=P$ for every $\lambda\in\GF(q)\setminus\{0\}$. Let $\lambda_0\in\GF(q)\setminus\{0\}$ such that $P\in C_{\lambda_0}$. Then $P\in P^{\perp_{\lambda_0}}=r_0$ and hence $P\in r_0$, a contradiction.
\end{proof}

	\begin{remark}
{\rm Notice that if $q$ is odd then the projectivity obtained as the product of two polarities associated to distinct conics of a circumscribed bundle is fixed point free.}	
	\end{remark}

Since ${\cal B}$ is stabilized by the normalizer $N$ of a Singer cyclic group $S$ of $G$ that is maximal in $G$ \cite{BHR} we get that
${\cal B}^G$ has size $q^3(q^2-1)(q-1)/3$.	
	
From \cite{RF} the group $G$ has three orbits on points and lines of $\PG(2,q^3)$. The orbits on points are the $q^2+q+1$ points of $\PG(2,q)$, the $(q^3-q)(q^2+q+1)$ points on lines of $\PG(2,q^3)$ that are not in $\PG(2,q)$ and the remaining set $E$ of $q^3(q^2-1)(q-1)$ points of $\PG(2,q^3)$.
The $G$--orbits on lines are the $q^2+q+1$ lines of $\PG(2,q)$, the $(q^3-q)(q^2+q+1)$ lines meeting $\PG(2,q)$ in a point and the
set $L$ of $q^3(q^2-1)(q-1)$ lines external to $\PG(2,q)$. The group $N$ fixes a triangle $T$ whose vertices are points of $E$ and whose edges are lines of $L$.

We will need the following lemma.

	\begin{lemma}\label{5arc}
Let $T^G$ be the orbit of $T$ under $G$. If $T_1$ and $T_2$ are distinct elements of $T^G$ then the union of their vertices always contains a $5$--arc, i.e.,  $5$ points of $\pi$ no three of which are collinear.	
	\end{lemma}
	\begin{proof}
The stabilizer of $T$ in $G$ contains $N$ that is maximal in $G$  and hence $\vert T^G\vert =q^3(q^2-1)(q-1)/3$. Let us consider the incidence structure whose points are the points of $E$ and whose blocks are the vertex sets of the triangles of $T^G$. The incidence relation is containment. It turns out that through a point of $E$ there pass exactly one triangle of $T^G$. Analogously,  let us consider the incidence structure whose points are the lines of $L$ and whose blocks are the edge sets of the triangles of $T^G$. The incidence relation is containment.  It turns out that through a line of $L$ there exists exactly one triangle of $T^G$ having that line as an edge.
As a consequence, the union of two triangles of $T^G$ always contains a $5$--arc of $\pi$.
	\end{proof}
	
	\begin{cor}
Two distinct circumscribed bundles of ${\cal B}^G$ share at most one conic.	
	\end{cor}	
	\begin{proof}
Let ${\cal B}_1$ and ${\cal B}_2$ be two distinct circumscribed bundles in ${\cal B}^G$. Let $T_i$ be the triangle associated to ${\cal B}_i$, $i=1,2$. By way of contradiction assume that $C_1$ and $C_2$ are distinct conics in ${\cal B}_1\cap{\cal B}_2$. Then $C_1$ and $C_2$, considered as conics of $\pi$,  contain the vertices of both triangles $T_1$ and $T_2$.  From Lemma \ref{5arc} and from \cite[Corollary 7.5]{JWPH1} we get a contradiction.
	\end{proof}

Under the Veronese map {\em v} the circumscribed bundles in ${\cal B}^G$ correspond to a set $\cal C$ of $q^3(q^2-1)(q-1)/3$ planes of $\PG(5,q)$ mutually intersecting in at most one point. Since no quadric in a bundle is degenerate, a plane of $\cal C$ is always disjoint from ${\cal S}$.

	\section{Two special webs of quadrics}
	
Firstly, we recall some basic properties of three--dimensional non--degenerate quadrics.

A {\em hyperbolic quadric} ${\cal Q}^+(3,q)$ of $\PG(3,q)$ consists of $(q+1)^2$ points of $\PG(3,q)$ and $2(q+1)$ lines that are the union of two reguli. A {\em regulus } is the set of lines intersecting three skew lines and has size $q+1$. Through a point of ${\cal Q}^+(3,q)$ there pass two lines belonging to different reguli. A plane of $\PG(3,q)$ is either secant to ${\cal Q}^+(3,q)$ and meets ${\cal Q}^+(3,q)$ in a conic or it is tangent to ${\cal Q}^+(3,q)$ and meets ${\cal Q}^+(3,q)$ in a bi--line.

An {\em elliptic quadric} ${\cal Q}^-(3,q)$ of $\PG(3,q)$ consists of $q^2+1$ points of $\PG(3,q)$ such that no three of them are collinear.
A plane of $\PG(3,q)$ is either secant to ${\cal Q}^-(3,q)$ and meets ${\cal Q}^-(3,q)$ in a conic or it is tangent to ${\cal Q}^-(3,q)$ and meets ${\cal Q}^-(3,q)$ in a point.  For more details on hyperbolic and elliptic quadrics in a three--dimensional projective space we refer to \cite{JWPH}.

Let $P_1,P_2$ be two distinct points of $\PG(2,q)$. Since $G$ is $2$--transitive on points of $\PG(2,q)$ we can always assume that $P_1=(1,0,0)$ and $P_2=(0,1,0)$. The set of quadrics of $\PG(2,q)$ passing through $P_1$ and $P_2$ are those having the coefficients $a_{11}=a_{22}=0$ and forms a web $W$. Under the Veronese map {\em v}, $W$ corresponds to the solid {\em v}$(W)$ with equations $X_1=X_2=0$. The solid {\em v}$(W)$ intersects ${\cal S}$ into the set of points satisfying the equations $X_4(2X_5X_6-X_3X_4)=0$ and $X_4(X_5X_6-X_3X_4)=0$ accordingly as $q$ is odd or even, respectively.
In both cases, this set consists of a hyperbolic quadric $Q$ and a plane tangent $\pi$ to $Q$ at the point $R=(0,0,1,0,0,0)$.
In particular, $\pi$ meets $Q$ at $2q+1$ points forming a bi--line centered at $R$. The point $R$ corresponds to the repeated line $P_1P_2$ and the remaining $2q$ points correspond to the bi--lines of $W$ centered at $P_1$ and $P_2$. It is easily seen that the number of such solids (hyperbolic solids) is $q(q+1)(q^2+q+1)/2$.

Assume that $P_1,P_2$ are points of $\PG(2,q^2)\setminus \PG(2,q)$ conjugate over $\GF(q)$. Since $G$ is transitive on points of $\PG(2,q^2)\setminus\PG(2,q)$ we can assume that $P_1=(1,\alpha,0)$ and so $P_2=(1,\alpha^q,0)$, where $\alpha$ is a primitive element of $\GF(q^2)$ over $\GF(q)$. Again, the set of quadrics of $\PG(2,q^2)$ passing through $P_1$ and $P_2$ are those whose coefficients satisfy $a_{11}=\alpha^{q+1}a_{22}$ and $a_{12}=-(\alpha+\alpha^q)a_{22}$ and forms a web $U$ . Under the Veronese map {\em v}, $U$ corresponds to the solid {\em v}$(U)$ with equations $X_1=\alpha^{q+1}X_2$ and $X_4=-(\alpha+\alpha^q)X_2$.
The solid {\em v}$(U)$ intersects ${\cal S}$ into the set of points satisfying the equations
$X_2(X_6^2+\alpha^{q+1}X_5^2+(\alpha+\alpha^q)X_5X_6+((\alpha+\alpha^q)^2-\alpha^{q+1})X_2X_3)=0$ and $X_2((\alpha+\alpha^q)^2X_2X_3+\alpha^{q+1}X_6^2+(\alpha+\alpha^q)X_5X_6+X_5^2)=0$ accordingly as $q$ is odd or even, respectively. Notice that the polynomial $X^2+(\alpha+\alpha^q)X+\alpha^{q+1}$ is irreducible over $\GF(q)$ and that, if $q$ is odd, $\alpha^{q+1}$ is a nonsquare element of $\GF(q)$.
Therefore, in both cases, this set consists of an elliptic quadric $Q'$ and a plane $\pi$ tangent to $Q'$ at the point $R=(0,0,1,0,0,0)$. In this case, the number of such solids (elliptic solids) is $q(q-1)(q^2+q+1)/2$.

The plane $\pi$ is contained in $\cal S$, belongs to $\cal T$ and meets ${\cal O}_1$ at the point $R$.

In the sequel a hyperbolic or elliptic solid will be denoted by $\Sigma=(\tau,Q)$ where $\tau\in{\cal T}$ is contained in $\Sigma$ and $Q$ is the three--dimensional hyperbolic or elliptic quadric contained in $\Sigma\cap{\cal S}$.
We will denote by $\cal H$ and ${\cal E}$ the set of hyperbolic solids and elliptic solids, respectively.

Now, we do investigate how two solids (elliptic or hyperbolic) can intersect.

	\begin{prop}\label{mcsh}
Let $\Sigma_1=(\pi_1,Q_1)$, $\Sigma_2=(\pi_2,Q_2)$ be two distinct hyperbolic solids. Then, one of the following cases occur:
\begin{itemize}
\item[1.] $\Sigma_1 \cap \Sigma_2$ is a plane, $\pi_1 = \pi_2$ and $\vert Q_1 \cap Q_2 \vert = q+1$;
\item[2.] $\Sigma_1 \cap \Sigma_2$ is a plane, $\pi_1 = \pi_2$ and $\vert Q_1 \cap Q_2 \vert = 1$;
\item[3.] $\Sigma_1 \cap \Sigma_2$ is a plane, $\vert \pi_1 \cap \pi_2 \vert = 1$ and $\vert Q_1 \cap Q_2 \vert = q+2$;
\item[4.] $\Sigma_1 \cap \Sigma_2$ is a line, $\vert \pi_1 \cap \pi_2 \vert = 1$ and $\vert Q_1 \cap Q_2 \vert = 2$;
\end{itemize} 	
	\end{prop}
	\begin{proof}

Let us assume that $\Sigma_i$ corresponds to the web defined by the points $A_i,B_i$, $i=1,2$. Let $\ell_i$ be the line $A_iB_i$, $i=1,2$.
\begin{itemize}
\item[1.] The pairs $A_1,B_1$ and $A_2,B_2$ share a point and $\ell_1=\ell_2$. Then we can assume that $A_2=B_1$. In this case
it is clear that $\pi_1=\pi_2$ and the $q+1$ points of $Q_1\cap Q_2$ correspond to the bi--lines centered at $A_2=B_1$ of the relevant webs together with $\ell_1=\ell_2$ considered as a repeated line.

\item[2.] The pairs $A_1,B_1$ and $A_2,B_2$ share no point and $\ell_1=\ell_2$. In this case it is clear that $\pi_1=\pi_2$
and the point of $Q_1\cap Q_2$ corresponds to $\ell_1=\ell_2$ considered as a repeated line.

\item[3.] The pairs $A_1,B_1$ and $A_2,B_2$ share the point $A_2=B_1=\ell_1\cap \ell_2$. In this case the planes $\pi_1$ and $\pi_2$ share only the point corresponding to the bi--line $\ell_1\ell_2$. On the other hand, $Q_1\cap Q_2$ contains the $q+1$ points corresponding to the bi--lines centered at a point of the line $A_1B_2$ and containing the line $A_1B_2$ and the line through $A_2$.
Also, $Q_1\cap Q_2$ contains the point corresponding to the bi--line $\ell_1\ell_2$.

\item[4.]  The pairs $A_1,B_1$ and $A_2,B_2$ share no point and $\ell_1\ne \ell_2$.
\begin{itemize}
\item[4.1] $\ell_1\cap\ell_2=A_2$.

In this case $\pi_1$ and $\pi_2$ share only the point corresponding to the bi--line $\ell_1\ell_2$. Here, $Q_1\cap Q_2$ consists of the two points corresponding to the bi--line $\ell_1A_1B_2$ centered in $A_1$ and the bi--line $\ell_1B_1B_2$ centered at $B_2$.
The line joining the points of $Q_1\cap Q_2$ lies on $\pi_1$.

\item[4.2]
The points $A_1,B_1,A_2,B_2$ form a $4$--arc in $\PG(2,q)$. In this case $\pi_1$ and $\pi_2$ share only the point corresponding to the bi--line $\ell_1\ell_2$. Here, $Q_1\cap Q_2$ consists of the two points corresponding to the bi--line containing the lines$A_1A_2$
and $B_1B_2$ and the bi--line containing $A_1B_2$ and $A_2B_1$.

\item[4.3]  $\ell_1\cap\ell_2=B_1$. In this case by switching $\ell_1$ and $\ell_2$ we are again in the case $4.1$.

\end{itemize}
\end{itemize}
	
	\end{proof}
	
	\begin{prop}\label{mcse}
Let $\Sigma_1=(\pi_1,Q_1)$, $\Sigma_2=(\pi_2,Q_2)$ be two distinct elliptic solids. One of the following cases occur:
\begin{itemize}
\item[1.] $\Sigma_1\cap\Sigma_2$ is a plane, $\pi_1=\pi_2$ and $\vert Q_1\cap Q_2\vert =1$;
\item[2.] $\Sigma_1\cap\Sigma_2$ is a line, $\vert\pi_1\cap\pi_2\vert =1$ and $\vert Q_1\cap Q_2\vert =2$;	
\end{itemize}	
	\end{prop}	
	\begin{proof}
Let us assume that $\Sigma_i$ corresponds to the web defined by the points $A_i,A_i^q$, $i=1,2$. Let $\ell_i$ be the line $A_iA_i^q$, $i=1,2$.	
\begin{itemize}
\item[1.] Assume that $\ell_1=\ell_2$. In this case it is clear that $\pi_1=\pi_2$ and that the unique intersection point between $Q_1$ and $Q_2$ corresponds to the repeated line $\ell_1=\ell_2$.

\item[2.] Assume that $\ell_1\ne \ell_2$.  In this case $\pi_1$ and $\pi_2$ share a unique point corresponding to the bi--line
$\ell_1\ell_2$. Here $Q_1\cap Q_2$ consists of the two points corresponding to the two imaginary bi--lines containing the lines
$A_1A_2$, $A_1^qA_2^q$ and $A_1A_2^q$, $A_1^qA_2$, respectively.
\end{itemize}	
	\end{proof}

	\section{Two special nets of quadrics}\label{nets}
	
As already observed, the Singer cyclic group  $S$ permutes the points (lines) of $\PG(2,q)$ in a single orbit.
Under the action of $S$,  the set of $q(q+1)(q^2+q+1)/2$ bi--lines of $\PG(2,q)$ is partitioned into $q(q+1)/2$ orbits of size $q^2+q+1$. Let us fix one of the $q(q+1)/2$ orbits of bi--lines, say $b$, and let us consider the incidence structure whose points are the lines of $\PG(2,q)$ and whose blocks are the bi--lines of $b$. It turns out that a line $\ell$ is contained in exactly two bi--lines, say $b_1$ and $b_2$, of $b$ centered at two distinguished points of $\ell$, say $A_1$ and $A_2$, respectively. Let $s$ be the unique element of $S$ such that $A_1^s = A_2$. Then $b_1^s = b_2$.

Let ${\cal P}_{A_1}$ and ${\cal P}_{A_2}$ be the pencils of lines with vertices $A_1$ and $A_2$.
Clearly,  $s$ is a projectivity sending ${\cal P}_{A_1}$ to ${\cal P}_{A_2}$ that does
not map the line $\ell$ onto itself. In \cite{steiner} it is proved that the set of points of
intersection of corresponding lines under $s$ is a conic $C$ passing through $A_1$ and $A_2$ (Steiner's argument).
The projectivity $s$ maps the tangent line to $C$ at $A_1$ onto the line $\ell$ and the line $\ell$ onto the
tangent line to $C$ at $A_2$. Moreover, for any two distinct points $A$ and $B$ of a conic there exists a projectivity $\psi\in S$ sending $A$ to $B$ and such that $C$ is the set of points of intersection of corresponding lines under $\psi$. Assume that $b_i=\ell\ell_i$, $i=1,2$. Since $s$ sends $\ell_1$ to $\ell$ and $\ell$ to $\ell_2$, it follows that $\ell_i$ is tangent to $C$ at $A_i$, $i=1,2$.
Embed $\PG(2,q)$ into $\PG(2,q^3)$. We have denoted by  $T$ be the unique triangle of $\PG(2,q^3)$ fixed by $S$. Considering
${\cal P}_{A_1}$ and ${\cal P}_{A_2}$ as pencils in $\PG(2,q^3)$ and repeating the previous argument, a conic $\bar C$ of $\PG(2,q^3)$ passing through the vertices of $T$ and containing $C$ arises. It follows that $C$ is a member of the circumscribed bundle $\cal B$ of $\PG(2,q)$ left invariant by
$S$. Let $A_3\in C\setminus\{A_1,A_2\}$ and let $b_3$ the bi--line of $b$ centered at $A_3$. Let $s'$ be the unique element of $S$ sending $A_1$ to $A_3$. Then $b_1^{s'}=b_3$. Steiner's argument with $s$ replaced by $s'$, applied to the pencils  ${\cal P}_{A_1}$ and ${\cal P}_{A_3}$, gives rise to a conic $C'$ that necessarily belongs to ${\cal B}$. Furthermore, being unique the conic of $\cal B$
through two distinct points of $\PG(2,q)$ it follows that $C=C'$.
Since $A_1^{s'}=A_3$ and $A_1\in\ell_1$ then $A_1^{s'}=A_3\in\ell_1^{s'}$. On the other hand, the point $\ell_1^{s'}\cap\ell_1$ lies
on $C$ and of course it lies on $\ell_1$. Since the line $\ell_1$ is tangent to $C$ at $A_1$ we have that $\ell_1^{s'}\cap\ell_1$ is the point $A_1$ or, in other words, $\ell_1^{s'}$ is the line $A_1A_3$. Analogously, the point $\ell^{s'}\cap\ell$ lies on $C$ and of course
it lies on $\ell$. Therefore $\ell^{s'}$ is either the line $A_1A_2$ or the line $A_1A_3$. Since $\ell_1\ne\ell$ it follows that $\ell^{s'}=A_2A_3$. We have showed that $b_3$ is the bi--line containing the lines $A_1A_3$ and $A_2A_3$.

We have proved the following Proposition.

	\begin{prop}\label{conics}
For any conic $C$ of ${\cal B}$ there exists two distinguished points $P_1$ and $P_2$ of $C$ such that the elements of $b$
centered at a point of $C$ are as follows: $t_{P_1}r$, $t_{P_2}r$, $r_1r_2$, where $t_{P_i}$ is the tangent line to $C$ at $P_i$, $i=1,2$,
$r$ is the line $P_1P_2$, $r_i$ is the line $PP_i$, $i=1,2$, and $P$ ranges over $C\setminus\{P_1,P_2\}$.
	\end{prop}
	
	\begin{remark}\label{corr}
{\rm Notice that there exists a one to one correspondence between the orbits of $S$ on bi--lines and secant lines to $C$.}	
	\end{remark}

\begin{remark}\label{conf}
{\rm With the notation introduced in Proposition \ref{conics}, notice that, from \cite[Table 7.7]{JWPH1} the pencil generated by the bi--lines $r_1r_2$ and $ru$, where $t_{P1}\cap t_{P_2}\in u$ and $P_1,P_2,P_3\not\in u$, contains exactly a further bi--line. Moreover, this bi--line is centered at a point of $C$. Indeed, let $C$  be the conic with equation $X_1X_3-X_2^2=0$. The stabilizer of $C$ in $G$ is isomorphic to $\PGL(2,q)$ and acts $3$--transitively on points of $C$.  Hence, without loss of generality, we can assume that
$P_1=(1,0,0)$, $P_2=(0,0,1)$ and $P_3=(1,1,1)$.  Let $v_i$ be the line joining the point $P_3$ and the point $P_i$, $i=1,2$. Then$b_3=v_1v_2$.  Notice that $U=t_{P_1}\cap t_{P_2}=(0,1,0)$.
Let $u$ be a line passing through $U$ and containing none of the points $P_i$, $i=1,2,3$. Let $b_4$ be the bi--line $uP_1P_2$. Then $b_3\cap b_4$ consists of the four points $P_1,P_2, (1,1,t), (1,t,t)$, with $t \ne 0,1$.
It turns out that the bi--line consisting of the lines $(1,t,t)P_2$ and $(1,1,t)P_1$ is centered at the point
$(1,t,t^2)\in C$.}
	\end{remark}

The following Proposition could be of some interest.	

\begin{prop}\label{pp}
The incidence structure whose points are the elements of $b$ and whose lines are the conics of the circumscribed bundle $\cal B$, where a bi--line is incident with a conic if it is centered at one of its points, forms a projective plane.
	\end{prop}
	\begin{proof}
We have that $\vert b \vert = \vert {\cal B} \vert = q^2+q+1$. Since through a point of $\PG(2,q)$ there pass $q+1$ conics of $\cal B$, we have that a bi--line of $b$ is incident with $q+1$ conics of $\cal B$. On the other hand a conic is incident with $q+1$ bi--lines of $b$. In particular we have seen that to a conic $C$ of $\cal B$ are associated two distinguished points $P, P^s \in C$, where $s \in S$ and all the bi--lines of $b$ incident with $C$ contain both $P, P^s$. Let us consider now a conic of ${\cal B} \setminus \{ C \}$. Then it is necessarily of the form $C^\mu$, for some non--trivial element $\mu \in S$. Then two possibilities occur according as one of the points $P^\mu, P^{s \mu}$ does belong to $C$ or does not. If the first case occurs then, assuming that $P^\mu$ is the point belonging to $C$, we have that $C \cap C^{\mu} = {P^\mu}$. If $t_{P}$ denotes the tangent line to $C$ at the point $P$, it turns out that $t_{P}^{\mu}$ is the tangent line to $C^{\mu}$ at the point $P^{\mu}$ and $t_{P}^{\mu} = PP^{\mu}$. Therefore the unique bi--line of $b$ incident with both $C$ and $C^{\mu}$ is centered at $P^{\mu}$. If the latter case occurs, then, by construction (a la Steiner) , $P P^{\mu} \cap P^{s} P^{s \mu} = P P^{\mu} \cap P^{s} P^{\mu s} = P P^{\mu} \cap (P P^{\mu})^{s}$ is the unique point in common between $C$ and $C^{\mu}$. Therefore the unique bi--line of $b$ incident with both $C$ and $C^{\mu}$ is centered at $C \cap C^{\mu}$.
\end{proof}

With the notation introduced in Proposition \ref{conics} let us consider the three bi--lines $t_{P_1}r$,  $t_{P_2}r$ and $r_1r_2$, where $r_1\cap r_2=P\in C\setminus\{P_1,P_2\}$. Under the map {\em v} they correspond to three points $R_1,R_2,R_3$ of ${\cal O}_2$,  respectively.
From the classification of pencils of quadrics of $\PG(2,q)$ in \cite[Table 7.7 ]{JWPH1} the line joining $R_1$ and $R_2$ corresponds to the unique pencil $\cal P$ whose members are all bi--lines and having a base consisting of $q+2$ points. Hence the line $R_1R_2$ is completely contained in ${\cal O}_2$. In particular, the bi--lines of $\cal P$ are those containing the line $r$ and the line $t_{P_1}\cap t_{P_2}A$, where $A$ ranges over $r$. It follows that the bi--line corresponding to $R_3$ cannot belong to $\cal P$.
Let {\em v}$(b)$ be the image of $b$ under {\em v}. Of course {\em v}$(b)$ contains $R_i$, $i=1,2,3$. Let $\pi_{e}$ be the plane of $\PG(5,q)$ generated by $R_1,R_2,R_3$ and let $\Pi_e$ denote the set of planes obtained in this way.

The plane $\pi_{e}$ meets ${\cal O}_2$ at $2q$ points consisting of the line $R_1R_2$ and of further $q-1$ points. Also, the plane $\pi_e$ meets $v(b)$ in $q+1$ points containing $R_1,R_2,R_3$. Indeed, from Remark \ref{conf}, through the point $R_3$ there are $q-2$ lines intersecting ${\cal O}_2$ in three points and $v(b)$ in two points.

It follows that the points of $\pi_{e}\cap{\cal O}_2$ correspond to the bi--lines of $b$ centered at points of $C$. We have that $\vert \Pi_e\vert =q(q+1)(q^2+q+1)/2$

On the other hand, the plane $\pi_{e}$ is contained in the hyperbolic solid defined by the points $P_1,P_2$
and then $\pi_{e}\cap {\cal O}_2$ consists of a conic and a line secant to it.

Since the number of hyperbolic solids equals $\vert \Pi_e\vert$  and each plane of $\Pi_e$ is contained in at least a hyperbolic solid, it follows that there exists a one to one correspondence between planes of $\Pi_e$ and hyperbolic solids.

Now, let $S'$ be the unique Singer cyclic group of $\PGL(3,q^2)$ containing $S$. It is clear that the circumscribed bundle ${\cal B}'$ of $\PG(2,q^2)$ fixed by $S'$ induces the circumscribed bundle ${\cal B}$ of $\PG(2,q)$ fixed by $S$.

Let $b_1'$ be the imaginary bi--line containing the lines $r,r^q$ and centered at the point $P\in\PG(2,q)$. Let $C$ be a conic of ${\cal B}$ through $P$ and let $\bar C$ be the unique conic of ${\cal B}'$ containing $C$.

Let $\bar b$ be the orbit of $b_1'$ under $S'$.  As already observed above there exist two points, say $P_1,P_2$ on $\bar C$ such that all elements of $\bar b$ centered at a point of $\bar C$ pass through $P_1$ and $P_2$. Also, since $P_1\cup P_2\in r\cup r^q$ and the tangent line to $C$ at $P$ is a line of $\PG(2,q)$,
it follows that $P_1,P_2\not\in\PG(2,q)$.
Under the action of $S$, $\bar b$ is partitioned into $q^2-q+1$ orbits of size $q^2+q+1$. Among these, we denote by $b'$ the unique $S$--orbit consisting of imaginary bi--lines.
It turns out that a member of $b'$ consists of the lines $z=RP_1$ and $z^q=RP_2$ for some $R\in C$.
Let $R_1,R_2\in C$, $R_1\ne R_2$.  Let $r_i=R_iP_1$ and $r_i^q=R_iP_2$, $i=1,2$. Since $r_1\cap r_2=P_1$ it follows that  $r_1^q\cap r_2^q=P_1^q$ and then $P_2=P_1^q$.

Notice that the line $P_1P_2$ arises from a line $a$ of $\PG(2,q)$ that is external to $C$. Let $A=t_{P_1}\cap t_{P_2}$, where $t_{P_1}$ and $t_{P_2}$ are the tangent lines to $\bar C$ at $P_1$ and $P_2$, respectively.
Then $A\in\PG(2,q)$. Indeed, when $q$ is odd, $A$ is the conjugate of $a$ with respect to $C$. When $q$ is even, $A$ is the nucleus of both $C$ and $\bar C$.

\begin{prop}\label{conics1}
For any conic $C$ of ${\cal B}$ there exists two distinguished points $P$ and $P^q$ of $\bar C$ not on $C$ such that the elements of $b'$ centered at a point of $C$ are of the form $XP,XP^q$, where $X$ ranges over $C$.
	\end{prop}
	
	\begin{remark}\label{corr1}
{\rm Notice that there exists a one to one correspondence between the orbits of $S$ on imaginary bi--lines and lines external to $C$}.
	\end{remark}
	
Similar arguments used in Proposition \ref{pp}  give the following result.
 	
\begin{prop}\label{pp1}
The incidence structure whose points are the elements of $b'$ and whose lines are the conics of the circumscribed bundle $\cal B$, where an imaginary bi--line is incident with a conic if it is centered at one of its points, forms a projective plane.
	\end{prop}

Let $d_i$ be the bi--line consisting of the lines $a$ and $D_iA$, $i=1,2$, where $D_1,D_2$ are distinct points of $a$.

With the notation introduced in Proposition \ref{conics1}, let us consider two bi--lines of the form $a,D_iA$, $i=1,2$ and the imaginary bi--line $XP, XP^q$, for some $X\in C$.
Under the map {\em v} they correspond to three points $R_1,R_2,R_3$, respectively. The points $R_1$ and $R_2$ are in $\cO_2$, whereas $R_3 \in \cO_3$.
From the classification of pencils of quadrics of $\PG(2,q)$ in \cite[Table 7.7 ]{JWPH1} the line joining $R_1$ and $R_2$ corresponds to the unique pencil $\cal P$ whose members are all bi--lines and having a base consisting of $q+2$ points. Hence the line $R_1 R_2$ is completely contained in ${\cal O}_2$. In particular, the bi--lines of $\cal P$ are those containing the line $a$ and the line $D A$, where $D$ ranges over $a$. Of course the imaginary bi--line corresponding to $R_3$ cannot belong to $\cal P$.
Let {\em v}$(b')$ be the image of $b'$ under {\em v}. Of course {\em v}$(b')$ contains $R_i$, $i=1,2,3$.
Let $\pi_{i}$ be the plane of $\PG(5,q)$ generated by $R_1,R_2,R_3$ and let $\Pi_i$ denote the set of planes obtained in this way.

The plane $\pi_{i}$ meets ${\cal O}_2$ in the line $R_1R_2$ and $\cO_3$ in further $q+1$ points. Indeed, from the classification of pencils of quadrics of $\PG(2,q)$ in \cite[Table 7.7 ]{JWPH1}, through the point $R_3$ there are $q$ lines intersecting ${\cal O}_3$ in two points and $\cO_2$ in one point. Each of these lines corresponds to the unique pencil consisting of $q-2$ conics, a bi--line and two imaginary bi--lines. Also, there exists a unique line through the point $R_3$ intersecting both $\cO_2$, $\cO_3$ in one point. Such a line corresponds to the unique pencil consisting of $q-1$ conics one bi--line and one imaginary bi--line.
It follows that the points of $\pi_{i}\cap{\cal O}_3$ correspond to the imaginary bi--lines of $b'$ centered at points of $C$. We have that $\vert \Pi_i\vert =q(q-1)(q^2+q+1)/2$

On the other hand, the plane $\pi_{i}$ is contained in the elliptic solid defined by the points $P,P^q$
and then $\pi_{i}\cap {\cal O}_3$ consists of a conic and $\pi_{i} \cap \cO_2$ of a line external to it.

Since the number of elliptic solids equals $\vert \Pi_i\vert$  and each plane of $\Pi_i$ is contained in at least an elliptic solid, it follows that there exists a one to one correspondence between planes of $\Pi_i$ and elliptic solids.
			
\section{Lifting Singer cycles}

Here, we assume that $q$ is odd.
From \cite{huppert}, we may assume that  $S$ is given by
$$
\left(
\begin{array}{ccc}
\omega & 0 & 0\\
0 &\omega^q & 0\\
0 & 0 & \omega^{q^2}
\end{array}
\right),
$$
where $\omega$ is a primitive element of $\GF(q^3)$ over $\GF(q)$. It follows that the lifting of $S$ to a collineation of $\PG(5,q)$
fixing the Veronese surface ${\cal O}_1$ has the following canonical form $A=diag(S^2,S^{q+1})$ \cite{BBCE}.

The group $\langle A\rangle$ has order $q^2+q+1$.
Geometrically, $\langle A\rangle$ fixes two planes of $\PG(5,q)$ , say $\rho_1$,
$\rho_2$, and partition the remaining points of $\PG(5,q)$ into
Veronese surfaces, \cite[Corollary 5]{BBCE}. In particular, the
planes $\rho_1$ and $\rho_2$ are both full orbits of $\langle
A\rangle$ and disjoint from the cubic hypersurface ${\cal S}$ \cite{BCS}.


From \cite{BBCE} the cubic hypersurface ${\cal S}$ is partitioned under $\langle A\rangle$ into Veronese surfaces.
The hypersurface ${\cal S}$ has $(q^2+1)(q^2+q+1)$ points and hence it consists of $q^2+1$ Veronese surfaces.

\section{The construction of subspace codes}

In this Section we prove our main result.

	\begin{lemma}
Two distinct planes of $\Pi_e$ can meet in at most one point. 	
	\end{lemma}
	\begin{proof}
Let $\sigma_1,\sigma_2$ be two distinct planes of $\Pi_e$. From Section \ref{nets} there exist uniquely determined hyperbolic solids $\Sigma_1=(\pi_1,Q_1)$ and $\Sigma_2=(\pi_2,Q_2)$ of $\cal H$ containing $\sigma_1$ and $\sigma_2$, respectively.
Let $c_i=\sigma_i\cap Q_i$ be the conic in $\Sigma_i$, $i=1,2$.

Assume first that $c_1,c_2$ belong to the same $S$--orbit. Then, from Proposition \ref{pp}, $c_1$ and $c_2$ share exactly one point.
Since $S$ permutes the planes of $\cal T$ in a single orbit we have that $\pi_1\ne\pi_2$. Therefore, from Proposition \ref{mcsh}, $Q_1\cap Q_2$ consists of either $2$ or $q+2$ points ($q+1$ points on a line together with a further point $Y$). Assume that $y=\sigma_1\cap\sigma_2$ is a line. If $\vert Q_1 \cap Q_2 \vert = 2$, then the conics $c_1$ and $c_2$ should share two points, a contradiction. If $\vert Q_1 \cap Q_2 \vert = q+2$, then it turns out that $c_1 \cap c_2 = \{Y\}$. On the other hand, since $y \subseteq \Sigma_1\cap\Sigma_2$, the line $y$ contains $Y$ and must be secant to both $Q_1$ and $Q_2$. Hence, again, the conics $c_1$ and $c_2$ should share two points, a contradiction.

Assume that $c_1,c_2$ do not belong to the same $S$--orbit. Then $c_1$ and $c_2$ have no point in common. Assume that $y=\sigma_1\cap\sigma_2$ is a line. If $Q_1 \cap Q_2$ consists of either $2$ or $q+1$ or $q+2$ points, then since $y \subseteq \Sigma_1\cap\Sigma_2$, from Proposition \ref{mcsh}, the conics $c_1$ and $c_2$ should share at least one point, a contradiction. If $\vert Q_1 \cap Q_2 \vert = 1$, since $y \subseteq \Sigma_1\cap\Sigma_2$, then the line $y$ either contains the point $Q_1 \cap Q_2$ and, again, the conics $c_1$ and $c_2$ should share one point, a contradiction, or the line $y$ is secant to both $c_1$, $c_2$ and $y \cap (c_1 \cup c_2)$ consists of four distinct points. If this last case occurs, then, under the inverse of the map {\em v}, these four points correspond to four distinct bi--lines having in common $q+1$ points of a line $z$ and a further point $Z \notin z$. In particular, let $c_i'$ denote the conic of the circumscribed bundle $\cal B$ locus of centers of the bi--lines corresponding to points of $c_i$, $i=1,2$. It turns out that $c_1' \ne c_2'$ (see Remark \ref{corr}) and $z$ is the polar line of the point $Z$ with respect to both $c_1'$ and $c_2'$, when $q$ is odd or $Z$ is the nucleus of both $c_1'$ and $c_2'$, when $q$ is even. But this contradicts Lemma \ref{bundle}.
	\end{proof}

	\begin{lemma}
Two distinct planes of $\Pi_i$ can meet in at most one point. 	
	\end{lemma}
	\begin{proof}
Let $\sigma_1,\sigma_2$ be two planes of $\Pi_i$. From Section \ref{nets} there exist uniquely determined elliptic solids $\Sigma_1=(\pi_1,Q_1)$ and $\Sigma_2=(\pi_2,Q_2)$ of $\cal E$ containing $\sigma_1$ and $\sigma_2$, respectively.
Let $c_i=\sigma_i\cap Q_i$ be the conic in $\Sigma_i$, $i=1,2$.

Assume first that $c_1,c_2$ belong to the same $S$--orbit. Then, from Proposition \ref{pp1}, $c_1$ and $c_2$ share exactly one point.	
Since $S$ permutes the planes of $\cal T$ in a single orbit we have that $\pi_1\ne\pi_2$. Therefore, from Proposition \ref{mcse}, $Q_1\cap Q_2$ consists of $2$ points. Assume that $y=\sigma_1\cap\sigma_2$ is a line, then the conics $c_1$ and $c_2$ should share two points, a contradiction. 	

Assume that $c_1,c_2$ does not belong to the same $S$--orbit. Then $c_1$ and $c_2$ have no point in common. Assume that $y=\sigma_1\cap\sigma_2$ is a line. If $Q_1 \cap Q_2$ consists of $2$ points, then, since $y \subseteq \Sigma_1\cap\Sigma_2$, from Proposition \ref{mcse}, the conics $c_1$ and $c_2$ should share at least one point, a contradiction. If $\vert Q_1 \cap Q_2 \vert = 1$, since $y \subseteq \Sigma_1\cap\Sigma_2$, then the line $y$ either contains the point $Q_1 \cap Q_2$ and, again, the conics $c_1$ and $c_2$ should share one point, a contradiction, or the line $y\subset \cO_2$ is external to both $c_1$, $c_2$. If this last case occurs, then, under the inverse of the map {\em v}, the points of $y$ correspond to bi--lines having in common $q+1$ points of a line $z$ and a further point $Z \notin z$. In particular let $c_i'$ denote the conic of the circumscribed bundle $\cal B$ locus of centers of the imaginary bi--lines corresponding to points of $c_i$, $i=1,2$. It turns out that $c_1' \ne c_2'$ (see Remark \ref{corr1})  and $z$ is the polar line of the point $Z$ with respect to both $c_1'$ and $c_2'$, when $q$ is odd or $Z$ is the nucleus of both conics $c_1'$ and $c_2'$, when $q$ is even. But this, again, contradicts Lemma \ref{bundle}.
\end{proof}

	\begin{theorem}
The set ${\cal C}\cup\Pi_i\cup\Pi_e\cup{\cal N}$ consists of $q^3(q^2-1)(q-1)/3+(q^2+1)(q^2+q+1)$ planes mutually intersecting in at most one point. 	
	\end{theorem}
	\begin{proof}
	\begin{itemize}
\item[1.] Assume that $\sigma_1\in {\cal C}$ and $\sigma_2\in \Pi_i\cup\Pi_e\cup{\cal N}$.
In this case $\sigma_1\subset {\cal O}_4$ and $\sigma_2$ always contains a line in ${\cal O}_2\cup {\cal O}_3$ and hence if $\sigma_1\cap\sigma_2$ was a line then $\sigma_1$ should contain a point of ${\cal O}_2\cup{\cal O}_3$.

\item[2.] Assume that $\sigma_1\in \Pi_i\cup\Pi_e$ and $\sigma_2\in{\cal N}$.
In this case $\sigma_2\subset {\cal S}$ whereas $\sigma_1$ meets $\cal S$ in the union of a conic and a line $r$. Hence if $\sigma_1\cap\sigma_2$ was a line, such a line should be $r$. From \cite[Table 7.7 ]{JWPH1} the line $r$ corresponds to the unique pencil of quadrics containing only bi--lines. On the other hand, a line of $\sigma_2$ corresponds to a pencil of quadrics always containing imaginary bi--lines or at most $q$ bi--lines.

\item[3.] Assume that $\sigma_1 \in \Pi_e$ and $\sigma_2 \in {\Pi_i}$.

Let $\Sigma_1=(\pi_1,Q_1)$ the unique hyperbolic solid of $\cal H$ containing $\sigma_1$ and let $\Sigma_2=(\pi_2,Q_2)$ the unique elliptic solid of $\cal E$ containing $\sigma_2$. Notice that $Q_1\setminus \pi_1$ is always disjoint from $Q_2\setminus \pi_2$.
Let $c_i=\sigma_i\cap Q_i$ be the conic in $\Sigma_i$, $i=1,2$. Assume that $\sigma_1\cap\sigma_2$ is a line $r$ Since $r\subset\Sigma_1\cap\Sigma_2$ it follows that $r\subset\pi_1\cap\pi_2$. Under the inverse of the map {\em v}, the points of $r$ correspond to bi--lines having in common $q+1$ points of a line $z$ and a further point $Z \notin z$. In particular, let $c_i'$ denote the conic of the circumscribed bundle $\cal B$ locus of centers of the (imaginary) bi--lines corresponding to points of $c_1$ ($c_2$). It turns out that $z$ is secant to $c_1'$ and external to $c_2'$. In particular, $z$ is the polar line of the point $Z$ with respect to both $c_1'$ and $c_2'$, when $q$ is odd or $Z$ is the nucleus of both conics $c_1'$ and $c_2'$, when $q$ is even. But this, again, contradicts Lemma \ref{bundle}.	
	\end{itemize}	
	\end{proof}
	
	\begin{cor}
There exists a constant dimension subspace code $\cal K$ with parameters $(6,q^3(q^2-1)(q-1)/3+(q^2+1)(q^2+q+1),4;3)_q$.	
	\end{cor}

	\begin{cor}
The code $\cal K$ admits a group of order $3(q^2+q+1)$ as an automorphism group. It is the normalizer of a Singer cyclic group
of $\PGL(3,q)$.
	\end{cor}
	
	\begin{remark}
{\rm We say that a constant dimension subspace code is {\em complete} if it is maximal with respect to set--theoretic inclusion.
Some computer tests performed with MAGMA \cite{magma} yield that our code is not complete when $q=3$. Indeed there exist
other $39$ planes that can be added to our set. However, when $q=4,5$ our code is complete. We conjecture that our code is complete whenever $q \ge 4$.}	
	\end{remark}


\begin{thebibliography}{10}

\bibitem{BBEF} R.D. Baker, J.M.N. Brown, G.L. Ebert, J.C. Fisher, Projective bundles, {\em Bull. Belg. Math. Soc. Simon Stevin} 1 (1994), no. 3, 329-336.

\bibitem{BBCE} R.D. Baker, A. Bonisoli, A. Cossidente, G.L. Ebert, Mixed partitions of $\PG(5,q)$, {\em Discrete Math.} 208/209 (1999), 23-29.

\bibitem{BCS} E. Ballico, A. Cossidente, A. Siciliano, External flats to varieties in symmetric product spaces over finite fields, {\em Finite Fields Appl.} 9 (2003), no. 3, 300-309.

\bibitem{BHR}  J.N. Bray, D.F.  Holt, Derek, C.M. Roney--Dougal, {\em The maximal subgroups of the low-dimensional finite classical groups}, London Mathematical Society Lecture Note Series, 407,Cambridge University Press, Cambridge, 2013.

\bibitem{magma} J. Cannon, C. Playoust, {\em An introduction to MAGMA}, University of Sydney,
Sydney, Australia, 1993.

\bibitem{CP1} A. Cossidente, F. Pavese, On subspace codes, {\em Designs Codes Cryptogr.} (to appear) DOI 10.1007/s10623-014-0018-6.


\bibitem{RF} R. Figueroa, A family of not $(V,l)$--transitive projective planes of order $q^3$, $q\not\equiv 1\mod 3$ and $q>2$, {\em Math. Z.} 181 (1982), no. 4, 471-479.

\bibitem{DGG} D.G. Glynn, {\em Finite projective planes and related combinatorial systems}, Ph.D. thesis, Adelaide Univ., 1978.

\bibitem{DGG1} D.G. Glynn, On Finite Division Algebras, {\em J. Combin. Theory Ser. A} 44 (1987), no. 2, 253-266.

\bibitem{HKK} T. Honold, M. Kiermaier, S. Kurz, Optimal binary subspace codes of length $6$, constant dimension $3$ and minimum distance $4$, {\em Contemp. Math.-Am. Math. Soc.} 632 (2015), 157-176.

\bibitem{JWPH1}  J.W.P. Hirschfeld, {\em Projective Geometries over Finite Fields},  Oxford Mathematical Monographs, Oxford Science Publications,The Clarendon Press, Oxford University Press, New York, 1998.

\bibitem{JWPH}  J.W.P. Hirschfeld, {\em Finite projective spaces of three dimensions},  Oxford Mathematical Monographs, Oxford Science Publications,The Clarendon Press, Oxford University Press, New York, 1985.

\bibitem{HT} J.W.P. Hirschfeld, J.A. Thas, {\em General Galois Geometries},  Oxford Mathematical Monographs, Oxford Science Publications,The Clarendon Press, Oxford University Press, New York, 1991.

\bibitem{huppert} B. Huppert, {\em Endliche Gruppen, I}, Die Grundlehren der Mathematischen Wissenschaften, Band 134 Springer-Verlag, Berlin-New York (1967).

\bibitem{steiner} J. Steiner, Systematische Entwicklung der Abh\"{a}ngigkeit geometrischer Gestalten von einander, Reimer,
Berlin (1832).



    \end{thebibliography}
    \end{document}